\documentclass[reqno]{amsart}

\usepackage{amsmath,amssymb,amscd, accents} \usepackage{graphicx}
\DeclareGraphicsExtensions{.eps} \usepackage{mathrsfs}
\usepackage[mathcal]{eucal}

\newtheorem{Thm}{Theorem}{\bfseries}{\itshape}
\newtheorem{Cor}{Corollary}{\bfseries}{\itshape}
\newtheorem{Prop}[Cor]{Proposition}{\bfseries}{\itshape}
\newtheorem{Lem}[Cor]{Lemma}{\bfseries}{\itshape}
{\bfseries}{\itshape}
{\bfseries}{\rmfamily}
\newtheorem{Ex}[Cor]{Example}{\scshape}{\rmfamily}
\newtheorem{Rem}[Cor]{Remark}{\scshape}{\rmfamily}

\renewcommand\ge{\geqslant} \renewcommand\le{\leqslant}
\let\tildeaccent=\~ \let\hataccent=\^
\renewcommand\~[1]{\widetilde{#1}}

\def\<{\left<} \def\>{\right>} \def\({\left(} \def\){\right)}

\def\abs#1{\left\vert #1 \right\vert}  

\let\parasymbol=\S \def\secref#1{\parasymbol\ref{#1}}

\let\polishL=l \def\Zoladek.{\.Zol\c adek}

 \def\Mat{\operatorname{Mat}}
 
 \def\Im{\operatorname{Im}}

\def\GL{\operatorname{GL}} \def\etc.{\emph{etc}.}

\def\:{\colon} \def\R{{\mathbb R}} \def\C{{\mathbb C}} \def\Z{{\mathbb
    Z}} \def\N{{\mathbb N}}  
  \def\S{\varSigma}
\def\l{\lambda}   
 
 \def\d{\,\mathrm d}

 \let\PolishL=\L \def\Lojas.{\PolishL ojasiewicz}

\def\rest#1{{\vert_{#1}}}

 \def\spec{\operatorname{Spec}}

\def\Diff{\operatorname{Diff}}
\def\TDiff{\operatorname{TDiff}}
\def\id{\operatorname{id}}
\def\Der{{\mathcal{D}}}
\def\fm{{\mathfrak m}}

\begin{document}

\title{Finiteness properties of formal Lie group actions}

\author{Gal Binyamini}\address{University of Toronto, Toronto, 
Canada}\email{galbin@gmail.com}
\thanks{The author was supported by the Banting
  Postdoctoral Fellowship and the Rothschild Fellowship}

\begin{abstract}
  Following ideas of Arnold and Seigal-Yakovenko, we prove that
  the space of matrix coefficients of a formal Lie group action
  belongs to a Noetherian ring. Using this result we extend the
  uniform intersection multiplicity estimates of these authors
  from the abelian case to general Lie groups. We also demonstrate
  a simple new proof for a jet-determination result of Baouendi. et
  al.

  In the second part of the paper we use similar ideas to prove a
  result on embedding formal diffeomorphisms in one-parameter groups
  extending a result of Takens. In particular this implies that the
  results of Arnold and Seigal-Yakovenko are formal consequence of our
  result for Lie groups.
\end{abstract}
\maketitle
\date{\today}

\section{Introduction}
\label{sec:intro}

Let $M$ be a differentiable manifold and $F:M\to M$ a smooth
diffeomorphism. The fixed points of $F$ play an important role in
determining its dynamical properties. More generally, a point $p\in M$
is said to be \emph{periodic} of period $k\in\Z$ if $F^k(p)=p$. The
asymptotic properties of the set $k$-periodic points are one of the
principal dynamical invariants of $F$.

Recall that the index of an isolated fixed point $p$ of $F$ is defined
to be the number of fixed points born from $p$ after a generic small
perturbation of $F$ (counted with signs according to orientation). The
number of $k$-periodic points counted with the corresponding indices
is a more natural topological invariant. In particular, it is homotopy
invariant and can be computed at the level of homology, for instance
using the Lefschetz fixed point formula.

Suppose a map $F$ is given such that the number of $k$-periodic
points, counted with indices, tends to infinity. In \cite{ShubSull},
Shub and Sullivan show that in this case the number of periodic points
of the map must be infinite as well. More specifically, they prove the
following.
\begin{Thm}[\protect{\cite[Main Proposition]{ShubSull}}]
  Suppose that $F:(\R^n,0)\to(\R^n,0)$ is a $C^1$ map and $0$ is an
  isolated fixed point of $F^k$ for every $k\in\N$. Then the index of
  $0$ as a fixed point of $F^k$ is bounded as a function of $k$.
\end{Thm}
Shub and Sullivan then apply this result to give a homological criterion
for the infinitude of the set of periodic points based on the Lefschetz
fixed point formula.

The index of an isolated fixed point may be viewed as a topological
intersection index. Indeed, if $\Delta\subset M\times M$ denotes the
diagonal manifold then the index of a fixed point $p$ is equal to the
intersection index of $\Delta$ and $(\id\times F)(\Delta)$ at $p$. It
is natural to ask whether an analog of Shub and Sullivan's main
proposition can be given for general intersection indices.

An affirmative answer in the holomorphic setting was provided by
Arnold in~\cite{Arnold}. Consider now $F:(\C^n,0)\to(\C^n,0)$ to be
a germ of a biholomorphism and let $V,W\subset(\C^n,0)$ be the germs
of two analytic submanifolds of complementary dimensions. Then $F^k(V)$
is again a germ of an analytic submanifold, and assuming that it has
an isolated intersection with $W$ at the origin, a local intersection
multiplicity $(F^k(V),W)$ is defined.

\begin{Thm}[\protect{\cite[Theorem 1]{Arnold}}]
  Suppose that every element of the sequence $\mu_k:=(F^k(V),W)$ for
  $k\in\N$ is finite. Then the sequence is uniformly bounded.
\end{Thm}

Arnold's proof relied on the rather nontrivial Skolem-Mahler-Lech
theorem from additive number theory. In~\cite{SeigalYakov}, Seigal and
Yakovenko showed that Arnold's result could be proved using simpler
and much more general Noetherianity arguments. Using this approach,
they were able to generalize the result to the case of groups of
formal diffeomorphisms generated by finitely many commuting
diffeomorphisms and vector fields (see~\secref{sec:notation} for the
definition of the formal diffeomorphism group $\Diff[[\C^n,0]]$ and
its action on formal schemes).

More specifically, a commutative group $G\subset\Diff[[\C^n,0]]$ is
said to be finitely generated if there exist finitely many formal
diffeomorphisms $F_1,\ldots,F_p$ and finitely many formal vector
fields $V_1,\ldots,V_q$ such that all of these generators commute in
the appropriate sense, and every element of $G$ can be written as a
finite product of maps $F_i^{\pm1}$ and $e^{tV_j}$ where $t\in\C$.
Alternatively, $G$ is a product of a (discrete) finitely generated
abelian subgroup of $\Diff[[\C^n,0]]$ and a connected
(finite-dimensional) abelian Lie subgroup of $\Diff[[\C^n,0]]$ such
that the two subgroups commute.

\begin{Thm}[\protect{\cite[Theorem 1]{SeigalYakov}}]
  Let $G\subset\Diff[[\C^n,0]]$ be a finitely generated commutative
  group. Let $V,W$ be two germs of formal schemes, and define
  \begin{equation}
    \mu_g : G\to\N, \qquad \mu_g = (g^*V,W).
  \end{equation}
  Then there exists a constant $N\in\N$ depending on $V,W$ and $G$
  such that
  \begin{equation}
    \forall g\in G: \quad \mu_g<\infty \implies \mu_g<N.
  \end{equation}  
\end{Thm}

Seigal and Yakovenko prove this theorem essentially by showing that
the set $G_j=\{g:\mu_g>j\}$ is given by the zero locus of an ideal
$I_j\subset R$, where $R$ is a certain Noetherian algebra of
continuous functions on $G$. The theorem (and its various
generalizations given in~\cite{SeigalYakov}) follow by a simple
Noetherianity argument.

In this paper we generalize the results of~\cite{SeigalYakov} from the
commutative case to the action of an arbitrary Lie group with finitely
many connected components. Moreover, we prove a formal embedding
result showing that arbitrary formal diffeomorphisms may be embedded
into the action of such groups, thus showing that the results
of~\cite{Arnold} and~\cite{SeigalYakov} follow formally from our
result. A synopsis of our results and some applications and
corollaries is presented in~\secref{sec:statement}. Proofs of the key
results and some further analysis is given in~\secref{sec:proofs}.

\subsection{Acknowledgements}

I would like to express my gratitude to Sergei Yakovenko for
introducing me to this problem. I would also like to thank Askold
Khovanskii and Pierre Milman for helpful discussions. Finally I wish
to thank the anonymous referees for numerous helpful comments, and
particularly for a suggested improvement to the presentation of the
proof of Theorem~\ref{thm:comm-gp-embedding}.

\section{Statement of our results}
\label{sec:statement}

\subsection{Notation}
\label{sec:notation}

The Lie groups considered in this paper are real Lie groups of finite
dimension, possibly admitting infinitely many connected components
(unless otherwise stated). Algebraic groups are always considered over
$\C$.

Let $\C[[x]]$ denote the ring of formal power series in $n$ variables.
Denote by $\fm\subset\C[[x]]$ the maximal ideal and by
$\C_p[[x]]:=\C[[x]]/\fm^{p+1}$ the ring of $p$-jets of formal power
series.

Let $\Diff[[\C^n,0]]$ denote the group of formal diffeomorphisms, i.e.
$\C$-algebra automorphisms of $\C[[x]]$ which map $\fm$ to itself. Let
$\Der[[\C^n,0]]$ denote the Lie algebra of formal singular vector
fields, i.e. $\C$-linear derivations of $\C[[x]]$ which map $\fm$ to
itself. Since $\Diff[[\C^n,0]]$ (resp. $\Der[[\C^n,0]]$) maps $\fm$ to
itself, there is an induced group of $p$-jets $\Diff_p[[\C^n,0]]$
(resp. Lie algebra $\Der_p[[\C^n,0]]$), and $\Diff[[\C^n,0]]$ (resp.
$\Der[[\C^n,0]]$) is the inverse limit of these groups (resp. Lie
algebras). When we talk about morphisms of Lie groups, algebraic
groups and Lie algebras into $\Diff[[\C^n,0]]$ or $\Der[[\C^n,0]]$ we
mean that the induced maps into the $p$-jet spaces are morphisms of
the prescribed type, for any $p\in\N$.

Recall that the formal exponentiation operator defines a map
\begin{equation}
  \exp:\Der[[\C^n,0]]\to\Diff[[\C^n,0]], \quad \exp(V)=e^V=I+V+\frac{V^2}{2}+\ldots
\end{equation}
For any $V\in\Der[[\C^n,0]]$ there is an associated one-parameter
group of formal diffeomorphisms $\tau_V:\R\to\Diff[[\C^n,0]]$ given by
$\tau_V(t)=e^{tV}$. Conversely, any analytic (or even continuous) one
parameter group $\tau:\R\to\Diff[[\C^n,0]]$ admits an infinitesimal
generator $V\in\Der[[\C^n,0]]$, i.e. $\tau=\tau_V$ (see \cite[Equation
3.6]{iy:book}).

For $\alpha\in\N^n$, let $x^\alpha\in\C[x]$ denote the corresponding
monomial in the standard multiindex notation. Let $(\cdot,\cdot)$
denote the standard inner product with respect to the monomial basis,
i.e. $(x^\alpha,x^\beta)=\delta_{\alpha,\beta}$.

\subsection{Noetherianity of Lie group actions}

Let $G$ denote a Lie group, and $\rho:G\to\Diff[[\C^n,0]]$ a
homomorphism of Lie groups. We say that $G$ acts formally on
$(\C^n,0)$. For $p\in\N$ we denote by $\rho_p:G\to\Diff_p[[\C^n,0]]$
the induced map.

\begin{Ex}
  Let $G=\R$ and $V\in\Der[[\C^n,0]]$. Then map
  $\tau_V:\R\to\Diff[[\C^n,0]]$ defined in~\secref{sec:notation}
  determines a formal action of $\R$ on $(\C^n,0)$. If $V$ is
  holomorphic in some disc containing the origin, then for $t\in\R$
  the action of $t$ is given by $e^{tV}$, the holomorphic time-$t$
  flow of $V$ in the usual analytic sense.
\end{Ex}

For any pair $\alpha,\beta\in\N^n$ we define the \emph{matrix
  coefficient}
\begin{equation}
  \rho_{\alpha,\beta}:G\to\C, \quad \rho_{\alpha,\beta}(g) = (\rho(g)(x^\alpha),x^\beta).
\end{equation}
We define \emph{the space of matrix coefficients of $\rho$} to be the
$\C$-vector space spanned by the set of all matrix coefficients of
$\rho$. This space agrees with the union of the corresponding spaces
of matrix coefficients of $\rho_p$ for $p\in\N$. Since the space of
matrix coefficients of a finite-dimensional group representation is
independent of the choice of basis elements, it follows that the space
of matrix coefficients of $\rho$ is invariant under a formal change of
coordinates (since it induces a linear isomorphism at the level of
jets).

We are now ready to state our main result. Denote by $C(G)$ the
$\C$-algebra of continuous functions on $G$.

\begin{Thm}\label{thm:main}
  Let $G$ be a Lie group with finitely many connected components and
  $\rho:G\to\Diff[[\C^n,0]]$ a Lie group homomorphism. Then the space
  of matrix coefficients of $\rho$ is contained in a Noetherian
  $\C$-algebra $R_\rho\subset C(G)$. Moreover, if $K$ is a maximal
  compact subgroup of $G$ then
  \begin{equation}
    \dim R_\rho\le \dim G+n(\dim G-\dim K)
  \end{equation}
  where $\dim R_\rho$ denotes the Krull dimension of $R_\rho$.
\end{Thm}

Theorem~\ref{thm:main} generalizes the Noetherianity result of
\cite{SeigalYakov} to the case of arbitrary Lie groups (with finitely
many connected components) acting formally on $(\C^n,0)$. The proof
of the theorem is presented in~\secref{sec:proof-main}.

\subsection{Local intersection dynamics}

By a \emph{germ of a formal scheme} we shall mean a closed subset
$V\subset\spec \C[[x]]$. In other words, $V$ will be identified with an ideal $I_V$
of the ring $\C[[x]]$. Two germs of formal schemes $V,W$ of complementary
dimension are said to intersect \emph{properly} if $I_V+I_W$ is an
ideal of finite codimension in $\C[[x]]$. In this case, we define their
intersection multiplicity at the origin to be
\begin{equation}
  (V,W) := \dim_\C \C[[x]]/(I_V+I_W)
\end{equation}
\begin{Rem}
  Note that this definition of multiplicity gives the correct notion
  when $V,W$ are complete intersections. In the general case a more
  delicate definition is needed to provide the proper notion of
  intersection multiplicities, cf. \cite[Section 8]{fulton:it}.
  However, for our purposes this naive definition will suffice, since
  it provides an upper bound for the true intersection multiplicity
  \cite[Proposition 8.2, (a)]{fulton:it}.
\end{Rem}

Since the ring $\C[[x]]$ is Noetherian, its ideals are finitely
generated. The following lemma is standard. We provide a proof for the
convenience of the reader.

\begin{Lem}[\protect{\cite[Lemma 3]{SeigalYakov}}]\label{lem:mult-cond}
  For any ideal $I\subset\C[[x]]$ and any $m\in\N$, the condition
  $\dim_\C \C[[x]]/I>m$ is equivalent to a finite number of algebraic
  conditions imposed on the $m$-jets of the generators of $I$.
\end{Lem}

\begin{proof}
  Denote by $j_m:\C[[x]]\to \C_m[[x]]$ the natural projection. We
  claim first that $\dim_\C \C[[x]]/I>m$ if and only if
  $\dim_\C \C_m[[x]]/j_m(I)>m$.

  If $\dim_\C \C_m[[x]]/j_m(I)>m$ then clearly $\dim_\C \C[[x]]/I>m$ as
  well. Conversely, suppose $\dim \C_m[[x]]/j_m(I)\le m$. Then
  $j_m(\fm^m)\subset j_m(I)$. Indeed, let
  $y=x_{i_1}\cdots x_{i_m}\in\fm^m$ where $1\le i_1,\ldots,i_m\le n$.
  Write $y_j=x_{i_1}\cdots x_{i_j}$ for $j=0,\ldots,m$. Since these are $m+1$ elements,
  they must be linearly dependent over $\C$ in $\C_m[[x]]/j_m(I)$,
  \begin{equation}
    c_0 y_0 + \cdots + c_m y_m \in j_m(I), \qquad c_0,\ldots,c_m\in\C.
  \end{equation}
  If $c_k$ is the first non-zero coefficient we have
  \begin{equation}
    y_k(c_k+c_{k+1}x_{i_{k+1}}+\cdots+c_m x_{i_{k+1}}\cdots x_{i_m}) \in j_m(I).
  \end{equation}
  The element in the parenthesis on the left hand side has a non-zero
  constant term, hence it is invertible in $\C_m[[x]]$, and we
  conclude that $y_k\in j_m(I)$ and certainly also $y\in j_m(I)$.
  Since this is true for any monomial $y\in\fm^m$ we have
  $j_m(\fm^m)\subset j_m(I)$ as claimed. It now follows by an
  $\fm$-completeness argument that $\fm^m\subset I$ as well
  \cite[Proposition~7.12]{Eis:book}. Then
  $\dim \C[[x]]/I=\dim \C_m[[x]]/j_m(I)\le m$ as claimed.

  Let $f_1,\ldots,f_g\in I$ denote a set of generators for $I$ and 
  \begin{equation}
    T:\C_m[[x]]^{\oplus g} \to \C_m[[x]], \qquad T(h_1,\ldots,h_g) = h_1f_1+\ldots+h_g f_g.
  \end{equation}
  Then $\Im T=j_m(I)$, hence the condition $\dim_\C \C_m[[x]]/j_m(I)>m$ is
  equivalent to the condition that the corank of $T$ is larger than
  $m$. Finally, one may of course express this condition by the vanishing
  of the appropriate minors of $T$, which may be viewed as algebraic
  equations involving the coefficients of $j_m(f_1),\ldots,j_m(f_g)$.
\end{proof}

Let $G$ be a Lie group and $\rho:G\to\Diff[[\C^n,0]]$ a Lie group
homomorphism. Then $G$ acts on germs of formal schemes: $g^*V$ is
defined to be the germ associated to the pull-back of the ideal $I_V$
by $\rho(g^{-1})$.

We can now state and prove our result on dynamics of intersections
for formal Lie group actions. The following is a relatively direct
consequence of Theorem~\ref{thm:main}.

\begin{Thm}\label{thm:group-intersection}
  Let $G$ be a Lie group with finitely many connected components and
  $\rho:G\to\Diff[[\C^n,0]]$ a Lie group homomorphism. Let $V,W$ be two
  germs of formal schemes, and define
  \begin{equation}
    \mu_g : G\to\N, \qquad \mu_g = (g^*V,W).
  \end{equation}
  Then there exists a constant $N\in\N$ depending on $V,W$ and $G$
  such that
  \begin{equation}
    \forall g\in G: \quad \mu_g<\infty \implies \mu_g<N.
  \end{equation}  
\end{Thm}
\begin{proof}
  Let $v_1,\ldots,v_s$ and $w_1,\ldots,w_t$ denote generators for the
  ideals $I_V$ and $I_W$ respectively. Let $R_\rho$ denote the algebra of
  Theorem~\ref{thm:main}. Then the functions mapping $g$ to the Taylor
  coefficients of $g^*v_i$ and $g^*w_i$ are matrix coefficients of
  $\rho$, hence elements of $R_\rho$. Then, by
  Lemma~\ref{lem:mult-cond} the set $\{g:(g^*V,W)>k\}$ is defined by
  the common zeros of an ideal $I_k\subset R_\rho$. Moreover, one can assume
  that $I_k\subset I_{k+1}$ (otherwise simply add to each ideal $I_k$
  the union of all $I_l$ for $l<k$).

  Since the algebra $R_\rho$ is Noetherian, the chain $I_k$ stabilizes
  at some finite index $k=N-1$. Then the set of common zeros of the
  ideals also stabilizes, i.e. for any $g\in G$ with finite $\mu_g$ we
  have in fact $\mu_g<N$.
\end{proof}

Theorem~\ref{thm:group-intersection} asserts the uniform boundedness
of the intersection multiplicity between a formal scheme $W$ and a
certain family of formal schemes $g^*V$, where $g$ varies over a Lie
group with finitely many connected components (possibly
\emph{non-compact}). We give two examples to illustrate that such a
result extends neither to arbitrary analytic families (even
one-dimensional), nor to arbitrary groups of diffeomorphisms (even
finitely generated).

\begin{Ex}
  We consider $\C^2$ with the coordinates $x,y$. Let $W:=\{y=0\}$
  and $V_t:=\{y=P_t(x)\}$ for $t\in\R$, where
  \begin{equation}
    P_t(x) = \sum_{j=1}^\infty (e^{\frac{\pi i t}j}-e^{-\frac{\pi i t}j}) x^j.
  \end{equation}
  Then $(V_t,W)$ is equal to the order of zero of $P_t$ at $x=0$. In
  particular, for $p\in\N$ prime it is easy to verify that
  $(V_{(p-1)!},W)=p$, and the multiplicities are therefore not
  uniformly bounded over $t\in\R$.

  The key feature in this example is that the set of exponents
  $e^{\l_j t}$ appearing in the expansion of $P_t(x)$ do not form a
  finitely-generated semigroup under multiplication. We shall see that
  this does not occur in Lie group actions.
\end{Ex}

\begin{Ex}[\protect{\cite[Example~4]{SeigalYakov}}]\label{ex:non-solvable}
  Let $g_1,g_2\in\Diff[[\C^1,0]]$ be two germs of formal
  diffeomorphisms with the expansions
  \begin{equation}
    g_j(x) := x+c_jx^{\nu_j+1}+\cdots, \qquad c_j\neq0,\, \nu_j\in\N,\, j=1,2. 
  \end{equation}
  If $\nu_1\neq\nu_2$ then a simple computation
  \cite[Proposition~6.11]{iy:book} shows that the commutator
  $g_3:=[g_1,g_2]$ admits a similar expansion with
  $\nu_3=\nu_1+\nu_2$. Repeating this consideration, we see that the
  group $G\subset\Diff[[\C^1,0]]$ generated by $g_1,g_2$ contains
  diffeomorphisms with a fixed point of arbitrarily high index.
  Interpreting the index of a fixed point as the corresponding
  intersection multiplicity with the diagonal as explained
  in~\secref{sec:intro}, one can use this to construct an example
  where the conclusion of Theorem~\ref{thm:group-intersection} fails.
\end{Ex}

\subsection{Finite jet determination for Lie group actions}

As another application of Theorem~\ref{thm:main}, we give an immediate
proof of a result of Baouendi et. al. concerning finite jet
determination in \cite[Theorem 2.10 and Proposition
5.1]{Baouendi:jet-determination}.

\begin{Thm}[\protect{\cite[Proposition 5.1]{Baouendi:jet-determination}}]
  Let $G$ be a Lie group with finitely many connected components
  and $\rho:G\to\Diff[[\C^n,0]]$ a continuous injective homomorphism.
  Then there exists $p\in\N$ such that, for any $g_1,g_2\in G$,
  \begin{equation}
    j^p\rho(g_1) = j^p\rho(g_2) \quad \text{if and only if} \quad g_1=g_2
  \end{equation}
\end{Thm}

We note that the statement of \cite[Proposition
5.1]{Baouendi:jet-determination} refers to analytic rather than formal
diffeomorphisms, and is implied by the statement above. We also note
that the original formulation works in the real (rather than complex)
category, but the proof is not affected by this difference.

\begin{proof}
  By considering $g=g_1g_2^{-1}$, it is clearly enough to prove
  the existence of $p\in\N$ such that for any $g\in G$,
  \begin{equation}
    j^p\rho(g) = \id \quad \text{if and only if} \quad g=e.
  \end{equation}
  Let $R_\rho$ denote the algebra of Theorem~\ref{thm:main}. Then
  the Taylor coefficients of $\rho(g)$ are elements of $R_\rho$,
  and in particular the condition $j^k\rho(g)=\id$ can be expressed
  as the common zero locus of an ideal $I_k\subset R_\rho$, where
  the chain $I_k$ is increasing.

  Since the algebra $R_\rho$ is Noetherian, the chain $I_k$ stabilizes at
  some finite index $k=p$. Then the set of common zeros of the
  ideals also stabilizes, i.e. the condition $j^p\rho(g)=\id$ is
  equivalent to the condition $\rho(g)=\id$. Since $\rho$ is an injection,
  the claim is proved.
\end{proof}

\subsection{Embedding in abelian Lie groups and formal flows}

In this subsection we state a theorem demonstrating that the results
of Arnold \cite[Theorem 1]{Arnold} and Seigal-Yakovenko \cite[Theorem
1]{SeigalYakov} are formal consequences of
Theorem~\ref{thm:group-intersection}. We note, to be clear, that in
the cases considered in these papers our proof essentially follows the
local computation of Arnold and the Noetherianity argument of
Seigal-Yakovenko. However, the embedding results given in this
subsection appear to be new and may be of some independent interest.

\begin{Thm}\label{thm:comm-gp-embedding}
  Let $\rho:\Z^p\times\R^q\to\Diff[[\C^n,0]]$ be a homomorphism of Lie
  groups. Then there exists an abelian Lie group $G$ with finitely
  many connected components (in fact, a linear algebraic group) and a
  homomorphism of Lie groups $\rho^*:G\to\Diff[[\C^n,0]]$ such that
  $\Im\rho\subset\Im\rho^*$.
\end{Thm}

The proof of Theorem~\ref{thm:comm-gp-embedding} is presented
in~\secref{sec:proof-main}.

\begin{Rem}
  We remark that \cite{SeigalYakov} uses a homomorphism of the
  \emph{complex} Lie group $\Z^p\times\C^q$, whereas we state
  Theorem~\ref{thm:comm-gp-embedding} for real Lie groups.
\end{Rem}

We record a simple corollary about embeddings of a diffeomorphism
in the flow of a formal vector field.

\begin{Cor} \label{cor-F-embedding}
  Let $F\in\Diff[[\C^n,0]]$ be a formal diffeomorphism. Then there
  exists some $k\in\N^+$ such that $F^k$ can be embedded in a formal
  flow, i.e. there exists a formal vector field $V\in\Der[[\C^n,0]]$
  such that $e^V=F^k$.
\end{Cor}
\begin{proof}
  Let $\rho:\Z\to\Diff[[\C^n,0]]$ be given by $\rho(1)=F$ and let
  $G$ and $\rho^*$ be as provided by Theorem~\ref{thm:comm-gp-embedding}.
  Let $G_0$ denote the connected component of the identity.

  By construction there exists $g\in G$ such that $\rho^*(g)=F$.
  Letting $k=\abs{G/G_0}$ we have $g^k\in G_0$. Since $G_0$ is an abelian
  connected Lie group, its exponential map is surjective. Thus there
  exists an analytic one-parameter group $\eta^*:\R\to G_0$ such that
  $\eta^*(1)=g^k$. Composing with $\rho^*$ we obtain an analytic one
  parameter group $\eta=\rho^*\circ\eta^*:\R\to\Diff[[\C^n,0]]$
  with $\eta(1)=F^k$. Then the derivative of $\eta$ at the origin
  gives a formal vector field $V$ satisfying $e^V=F^k$ as claimed.
\end{proof}

We note that Corollary~\ref{cor-F-embedding} is purely formal: even if
one starts with a holomorphic diffeomorphism $F$, we only guarantee
the existence of a formal vector field $V$ exponentiating to $F^k$.
The problem of embedding a holomorphic diffeomorphism in the flow of a
\emph{holomorphic} vector field is significantly more delicate and
admits a non-trivial analytic obstruction. For instance, a germ of a
holomorphic diffeomorphism $F:(\C,0)\to(\C,0)$ tangent to identity is
embeddable in a holomorphic flow if and only if its Ecalle-Voronin
modulus vanishes \cite[Theorem~21.31]{iy:book}.

\begin{Rem}
  The anonymous referee has mentioned that~\cite{martinet} constructs
  a Jordan decomposition for a formal vector field using the usual
  Jordan decomposition at the level of $k$-jet. A similar construction
  can be used to decompose the formal diffeomorphism $F$ in
  Corollary~\ref{cor-F-embedding} as a product $F=HG$ where $H$ is
  semi-simple, $G$ is unipotent, and $H$ and $G$ commute. This may
  provide an alternative approach to the proof of
  Corollary~\ref{cor-F-embedding}.
\end{Rem}

In~\secref{sec:embedding-proof} a more accurate description of the
group $G$ is provided. In particular we show how to compute $k$ above
in terms of the spectrum of the linear part of $F$, and use this to
give a new proof of a result of Takens \cite{Takens} about embeddings
of formal diffeomorphisms with a unipotent linear part.

\section{Proofs}
\label{sec:proofs}

In this section we prove the two theorems stated without proof
in~\secref{sec:statement}, namely
Theorems~\ref{thm:main}~and~\ref{thm:comm-gp-embedding}.

\subsection{Proof of Theorem~\ref{thm:main}}
\label{sec:proof-main}

Before presenting the proof we require two results from the theory
of Lie groups. The first is a decomposition theorem due independently
to Iwasawa \cite{iwasawa} and Malcev \cite{malcev}.

\begin{Thm}[\protect{\cite[Theorem 6]{iwasawa}}]\label{thm:iwasawa}
  Let $G$ be a connected Lie group, and $K$ a maximal compact subgroup.
  Then there exist subgroups $H_1,\ldots,H_r\subset G$ isomorphic
  to $\R$ such that any element $g\in G$ can be decomposed uniquely and
  continuously in the form
  \begin{equation}
    g=h_1\cdots h_r k \qquad h_i\in H_i \qquad k\in K.
  \end{equation}
\end{Thm}

The following linearization theorem is due to Bochner. Since we are not
aware of a reference for the formal case, we present the proof below.

\begin{Thm}[\protect{\cite[Theorem~1]{bochner}}]\label{thm:bochner}
  Let $K$ be a compact Lie group and $\rho:K\to\Diff[[\C^n,0]]$ a homomorphism
  of Lie groups. Then after a formal change of coordinates, $\rho$ is equivalent
  to its linear representation (mapping $g$ to the linear part $\d\rho(g)$).
\end{Thm}

\begin{proof}
  Let $\mu$ denote the normalized left-invariant Haar measure on $K$.
  We define a formal map $U\in\C[[x]]^n$ by averaging,
  \begin{equation}
    U = \int_K (\d\rho(g))^{-1}\rho(g)\d\mu(g). 
  \end{equation}
  Taking linear parts we see that
  \begin{equation}
    \d U = \int_K (\d\rho(g))^{-1}\d\rho(g) = \id,
  \end{equation}
  so $U\in\Diff[[\C^n,0]]$. Moreover, by invariance of the Haar measure,
  \begin{equation}
    \begin{split}
      U\circ\rho(h) &= \int_K (\d\rho(g))^{-1}\rho(gh)\d\mu(g) \\
      &= \int_K (\d\rho(gh^{-1}))^{-1}\rho(g)\d\mu(g) = \d\rho(h) U
    \end{split}
  \end{equation}
  where in the middle equality we translate $g\to gh^{-1}$, using the
  invariance of the Haar measure. We thus see that $U$ conjugates
  $\rho(h)$ to $\d\rho(h)$ as claimed.
\end{proof}

We require one more lemma. While simple, this lemma is in fact the
heart of the proof. The lemma is due to Arnold \cite[Lemma 2]{Arnold}
(see also \cite[Lemma 5]{SeigalYakov}).

\begin{Lem}\label{lem:exponent-coefficients}
  Let $\rho:\R\to\Diff[[\C^n,0]]$ be given by $\rho(t)=e^{tV}$ where
  $V\in\Der[[\C^n,0]]$. Denote by $\l_1,\ldots,\l_n$ the spectrum
  of the linear part of $V$. Denote by $t$ the coordinate
  on $\R$.
  
  Then the matrix coefficients of $\rho$
  are contained in the ring $\C[e^{\l_1t},\ldots,e^{\l_nt},t]$.
\end{Lem}
\begin{proof}
  The matrix coefficients of $\rho$ are the union of the matrix
  coefficients of the jets $j^p\rho$ for all $p\in\N$. It is a
  classical result, following easily from the Jordan decomposition,
  that the matrix coefficients of the matrix exponential $e^{tL}$ of a
  finite dimensional operator $L$ are contained in the ring
  $\C[e^{\mu_it},t]$ where $\mu_i$ are the elements of the spectrum of
  $L$. Thus the claim will be proved if we show that the spectrum of
  the jet $j^p\rho$ is contained in the semigroup generated by
  $\l_1,\ldots\l_n$ for every $p$.

  We may, after a linear change of coordinates, assume that the linear
  part of $V$ is a lower-triangular matrix. Put the
  degree-lexicographic ordering on all monomials in $\C_p[[x]]$. Then
  the matrix representing $j^p V$ is again lower-triangular, and the
  diagonal entry corresponding to the monomial $x^\alpha$ is precisely
  $\sum \alpha_i \lambda_i$. Thus the spectrum is contained in the
  semigroup generated by the $\lambda_i$, as claimed.
\end{proof}

We are now ready to present the proof of Theorem~\ref{thm:main}.
\begin{proof}[Proof of Theorem~\ref{thm:main}]
  Recall that $K$ denotes a maximal compact subgroup of $G$. Assume
  first that $G$ is connected. By Theorem~\ref{thm:iwasawa} we have
  for every $g\in G$ a decomposition
  \begin{equation} \label{eq:g-decomp}
    g=h_1\cdots h_r k \qquad h_i\in H_i \qquad k\in K 
  \end{equation}
  where we may think of $h_i,k$ as continuous functions of $g$.

  By Theorem~\ref{thm:bochner}, making a formal change of coordinates
  we may assume that the action of $K$ is linear (recall that a formal
  change of coordinates does not change the space of matrix
  coefficients). Denote by $\~K\subset\GL(n,\C)$ the image of $K$
  under $\rho$. Then $\~K$ is a real algebraic subgroup by a theorem
  of Chevalley, and clearly $\dim\~K\le\dim K$. Let $\~K_\C$ denote
  the complexification of $\~K$, which is a complex-algebraic group of
  complex dimension equal to $\dim\~K$. Since the action of $K$ is
  linear, the space of matrix coefficients of $\rho\rest K$ is contained in the
  finitely-generated $\C$-algebra of regular functions $R[\~K_\C]$ on
  the algebraic group $\~K_\C$.

  For $i=1,\ldots,r$ the subgroup $H_i$ is isomorphic to $\R$, and its
  action $\rho_i:H_i\to\Diff[[\C^n,0]]$ is therefore given by
  $\rho_i(t)=e^{tV_i}$ for some infinitesimal generator
  $V_i\in\Der[[\C^n,0]]$ where $t\in\R\simeq H_i$. By
  Lemma~\ref{lem:exponent-coefficients} the matrix coefficients of
  $\rho_i$ belong to a ring $R_i$, which is a finitely-generated
  $\C$-algebra of dimension at most $n+1$.

  Now,
  \begin{equation}
    \rho(g) = \rho(h_1)\cdots\rho(h_r)\rho(k)
  \end{equation}
  so the matrix coefficients of $\rho(g)$ are multilinear combinations
  of the matrix coefficients of $\rho(h_i),\rho(k)$. In particular they belong
  to the ring
  \begin{equation}
    R_\rho = R_1\otimes_\C \cdots \otimes_\C R_r \otimes_\C R[\~K_\C]
  \end{equation}
  which, by the above, is a Noetherian $\C$-algebra (in fact a
  finitely generated $\C$-algebra) of dimension bounded by $\dim G+nr$
  as claimed.

  For the non-connected case, one can carry out the proof above in
  the same manner, noting that Theorem~\ref{thm:iwasawa} remains
  true for Lie groups with finitely many connected components (see
  for instance \cite[Chapter~XV]{hochschild}).
\end{proof}

\subsection{Proof of Theorem~\ref{thm:comm-gp-embedding} and further results}
\label{sec:embedding-proof}

In this subsection we prove Theorem~\ref{thm:comm-gp-embedding} and
give some more detailed results about embeddings of formal
diffeomorphisms in formal flows.

We begin with a slightly more detailed analog of
Lemma~\ref{lem:exponent-coefficients} for the discrete case. 

\begin{Lem}\label{lem:power-coefficients}
  Let $F\in\Diff[[\C^n,0]]$ and denote by $\l_1,\ldots,\l_n$ the
  spectrum of the linear part of $F$. Let $\tau:\Z\to\Diff[[\C^n,0]]$
  be given by $\tau(1)=F$. Denote by $t$ the coordinate on $\Z$.

  Then:
  \begin{enumerate}
  \item If $j^p F$ is a semisimple operator for every $p\in\N$ then
    the space of matrix coefficients of $\tau$ is equal to
    $\C[\l_1^t,\ldots,\l_n^t]$.
  \item Otherwise, the space of matrix coefficients of $\tau$ is
    contained in $\C[\l_1^t,\ldots,\l_n^t,t]$ and contains a function
    of the form
    $\l^{\alpha t}t=\l_1^{\alpha_1t}\cdots\l_n^{\alpha_n t} t$ for
    some $\alpha\in\N^n$.
  \end{enumerate}
\end{Lem}

The proof, which we omit, proceeds in the same way as the proof of
Lemma~\ref{lem:exponent-coefficients}. One merely needs to replace the
computation of $e^{tL}$ for a matrix $L$ in Jordan form by the equally
well-known computation of $L^t$ illustrated in the following simple
example.

\begin{Ex}
  Let $A\in\Mat_{3\times3}(\C)$ be given by
  \begin{equation}
    A= \begin{pmatrix}
      \l_1&1& \\
      &\l_1& \\
      &&\l_2
    \end{pmatrix}.
  \end{equation}
  Then for $t\in\Z$,
  \begin{equation}
    A^t = \begin{pmatrix}
      \l_1^t& t\l_1^{t-1}& \\
      &\l_1^t& \\
      &&\l_2^t
    \end{pmatrix}
  \end{equation}
  and we see that the space of matrix coefficients is contained in
  $\C[\l_1^t,\l_2^t,t]$ and contains the function $\l_1^t t$.
\end{Ex}

Denote by $\TDiff[[\C^n,0]]$ the group of formal diffeomorphisms whose
linear part is lower-triangular with respect to the standard
coordinates. We remark that, in the same way as in the proof of
Lemma~\ref{lem:exponent-coefficients}, one can check that such
diffeomorphisms are automatically lower-triangular with respect to the
degree-lexicographic ordering on the standard monomial basis. We
denote by $\TDiff_p[[\C^n,0]]$ the corresponding group of jets. For
natural numbers $q>p$ denote by
$\pi_{q,p}:\TDiff_q[[\C^n,0]]\to\TDiff_p[[\C^n,0]]$ the natural
projection.

For any $q\in\N$ we denote by $R_q^T:=R[\TDiff_q[[\C^n,0]]$ the ring of
coordinates of $\TDiff_q[[\C^n,0]]$. For any pair of multindices
$|\alpha|,|\beta|\leq$ with $\alpha\ge\beta$ (in lexicographic order)
we have a coordinate function
\begin{equation}
  X_{\alpha,\beta}:\TDiff_q[[\C^n,0]]\to\C, \qquad X_{\alpha,\beta}(H):= (Hx^\alpha,x^\beta).
\end{equation}
For $j=1,\ldots,n$ we denote $L_j:=X_{j,j}$ the corresponding
coordinate on the diagonal. Then $R_q^T$ is generated as a ring by
$X_{\alpha,\beta}$ and $L_1^{-1},\ldots,L_n^{-1}$. Indeed, it is a
closed subgroup of the group of all triangular matrices in $\C_q[[x]]$
(with respect to the standard monomial basis) whose coordinate ring is
generated by $X_{\alpha,\beta}$ and the inverses of all diagonal
elements $X_{\alpha,\alpha}$. It remains only to note that it suffices
to localize by $L_{1\ldots n}$ because in $\TDiff_q[[\C^n,0]]$ we have
$X_{\alpha,\alpha}=L^\alpha$.

\begin{Lem} \label{lem:Gp-Gq}
  Let $F\in\TDiff[[\C^n,0]]$. Let $p$ denote the first index such
  that $j^pF$ is non-semisimple, or $1$ if no such index exists.
  For any $k\in\N$ denote by $G^*_k\subset\TDiff_k[[\C^n,0]]$
  the group generated by $j^k(F)$, and by $G_k$ its Zariski closure
  (i.e. the algebraic group generated by $j^k(F)$).

  Then for any integer $q\ge p$ the projection $\pi_{q,p}$ restricts
  to an isomorphism of algebraic groups $G_q\to G_p$.
\end{Lem}

\begin{proof}
  We consider the case that $j^pF$ is non-semisimple. The other case
  is similar but simpler. Denote by $R_p,R_q$ the coordinate rings of
  $G_p,G_q$ respectively. Denote by $\tau_p:\Z\to G_p$ the group
  homomorphism satisfying $\tau_p(1)=F$, and similarly for $\tau_q$.

  Since $\pi_{q,p}G^*_q=G^*_p$ by definition, $\pi_{q,p}$ restricts to
  a homomorphism of algebraic groups $G_q\to G_p$. Moreover since
  $G^*_p$ is dense in $G_p$ this map is dominant. Thus we have an
  induced injection $\iota:R_p\to R_q$, and we claim that it is also a
  surjection. Henceforth we identify $R_p$ with its image in $R_q$.
  The image of $\iota$ contains $X_{\alpha,\beta}$ with
  $|\alpha|,|\beta|\le p$ and $L_1^{-1},\ldots,L_n^{-1}$ by
  definition. We will show that it also contains $X_{\alpha,\beta}$
  for $|\alpha|,|\beta|\le q$ thus concluding the proof. Let
  $\alpha,\beta$ be such a pair.

  Denote by $\l_1,\ldots,\l_n$ the diagonal elements of the linear
  part of $F$. We clearly have
  \begin{equation}
    L_i(j^p F^t) = \l_i^t, \qquad \forall t\in\Z.
  \end{equation}
  We claim that there also exists a regular function $S\in R_p$
  such that
  \begin{equation}
    S(j^p F^t) = t, \qquad \forall t\in\Z.
  \end{equation}
  Indeed, by Lemma~\ref{lem:power-coefficients}, $\l^{\alpha t}t$
  appears in the space of matrix coefficients of $\tau_p$. It is
  therefore given by a linear combination of matrix coefficients. That
  is, there exists a regular function $S^*\in R_p$ such that
  $S^*(j^pF^t)=\l^{\alpha t}t$. We may thus define $S:=S^*/L^\alpha$.

  Recall that by Lemma~\ref{lem:power-coefficients} the space of
  matrix coefficients of $\tau_q$ is contained in
  $\C[\l_1^t,\ldots,\l_n^t,t]$. Thus there exist polynomials
  $P_{\alpha,\beta}$ satisfying
  \begin{equation}
    X_{\alpha,\beta}(j_q F^t) = P_{\alpha,\beta}(\l_1^t,\ldots,\l_n^t,t) = (P_{\alpha,\beta}(L_1,\ldots,L_n,S))(j_q F^t)
  \end{equation}
  In other words, $X_{\alpha,\beta}$ agrees with
  $P_{\alpha,\beta}(L_1,\ldots,L_n,S)$ on $G^*_q$, and hence also on
  $G_q$, i.e. as elements of $R_q$. Thus $X_{\alpha,\beta}$ is indeed
  in the image of $\iota$ as claimed.
\end{proof}

\begin{Rem}
  The conclusion of Lemma~\ref{lem:Gp-Gq} fails if one replaces the
  cyclic group generated by $F$ by a group $G$ generated by two
  arbitrary diffeomorphisms. For instance, in
  Example~\ref{ex:non-solvable} we presented a group $G$ with two
  generators containing diffeomorphisms with fixed points of
  arbitrarily high index. Since a diffeomorphism with a fixed point of
  index $q>1$ corresponds to a non-trivial element in the kernel of
  $\pi_{q,q-1}$ we see that the conclusion of Lemma~\ref{lem:Gp-Gq}
  cannot hold for any $p\in\N$.

  On the other hand, Lemma~\ref{lem:Gp-Gq} may still hold (for an
  appropriate choice of $p\in\N$) under more restrictive assumptions,
  for instance for commutative finitely-generated groups. One may
  attempt to generalize the proof by obtaining an analogous
  description of the rings generated by the matrix coefficients of
  such groups. We do not pursue this further in this paper, and use an
  alternative approach to prove the embedding result for groups with
  several generators.
\end{Rem}

We now state the key proposition on the embedding of a formal
diffeomorphism in an abelian linear algebraic group.

\begin{Prop} \label{prop:diff-embed}
  Let $F\in\Diff[[\C^n,0]]$. Then there exists an abelian linear
  algebraic group $G$ and a homomorphism of Lie groups $\rho:G\to\Diff[[\C^n,0]]$
  such that $F\in\Im\rho$.

  Additionally, if a formal diffeomorphism $F'\in\Diff[[\C^n,0]]$
  commutes with $F$ then it also commutes with any element of $\Im\rho$.
\end{Prop}
\begin{proof}
  After a linear change of coordinates we may assume that $F$ belongs
  to $\TDiff[[\C^n,0]]$. We may thus apply Lemma~\ref{lem:Gp-Gq}. In
  the notations of the lemma, for any $q\ge p$ we denote
  $\rho_q:G_p\to G_q$ the isomorphism of algebraic groups inverse to
  $\pi_{q,p}$. If $q'\ge q\ge p$ then
  $\pi_{q',q}\circ\rho_{q'}=\rho_q$. Thus $\rho_q:q\ge p$ form an
  inverse system with a limit $\rho:G_p\to \TDiff[[\C^n,0]]$, which is
  a homomorphism of Lie groups by definition. Moreover, since
  $\rho_q(j^p(F))=j^q(F)$ for any $q\ge p$, we have $\rho(j^p(F))=F$.
  The claim thus follows with $G=G_p$.

  Finally, suppose $F'\in\Diff[[\C^n,0]]$ commutes with $F$. Then for
  any $q\ge p$ the jet $j^q F'$ commutes with $j^q F$, hence with
  $G^*_q$, and hence also with $G_q$. Since
  $G_q=\Im \rho_q=j^q(\Im\rho)$ we see that $F'$ commutes with
  $\Im\rho$ at the level of $q$-jets for any $q\ge p$, so it in fact
  commutes with $\Im\rho$ as claimed.
\end{proof}



\begin{Rem} \label{rem:G-structure} From the proof above it is easy to
  find the structure of the group $G$. Denote by $\C_a,\C_m$ the
  additive and multiplicative groups of $\C$ respectively and by
  $D\subset\GL(n,\C)$ the group of diagonal matrices. Let $F_s$ denote
  the semisimple part of the linear part of $F$. After a linear change
  of coordinates we may assume $F_s\in D$.

  With these notations $G$ is isomorphic to either $T$ or
  $T\times\C_{\text{a}}$ (for the semisimple and non-semisimple cases
  in Lemma~\ref{lem:Gp-Gq}, respectively) where $T\subset D$ is the
  algebraic group generated by $F_s$. Denote by $T_0$ the connected
  component of the identity in $T$. The number of connected components
  of $G$ is given by the size of the cyclic group $T/T_0$, which can
  be expressed in terms of the spectrum $\l_1,\ldots,\l_n$ of $F_s$ as
  follows. Let $\Phi$ denote the group of multiplicative characters
  $\chi:D\to\C_m$ such that $\chi(F_s)$ is a root of unity. Then $T_0$
  is the common kernel of the characters in $\Phi$, and consequently
  \begin{equation}\label{eq:T-conn-comp}
    \#(T/T_0) = \#\{\chi(F_s):\chi\in\Phi\}.
  \end{equation}
  Finally, the right hand side of~\eqref{eq:T-conn-comp} is the size
  of the group of roots of unity inside the multiplicative group
  generated by $\l_1,\ldots,\l_n$.
\end{Rem}

Theorem~\ref{thm:comm-gp-embedding} is now an easy corollary.

\begin{proof}[Proof of Theorem~\ref{thm:comm-gp-embedding}]
  We prove the claim by induction on $p$. Let
  $\rho:\Z^p\times\R^q\to\Diff[[\C^n,0]]$ be a homomorphism of Lie
  groups. By the inductive hypothesis, there exists an abelian linear
  algebraic group $G$ and a homomorphism of Lie groups
  $\rho^*:G\to\Diff[[\C^n,0]]$ such that $\rho(\{0\}\times\Z^{q-1}\times\R^q)\subset\Im\rho^*$, and
  $\rho(1,0,\ldots,0)$ commutes with $\Im\rho^*$.

  By Proposition~\ref{prop:diff-embed} there exists an abelian linear
  algebraic group $G'$ and a homomorphism of Lie groups $\rho':G'\to\Diff[[\C^n,0]]$
  such that $\rho(1,0,\ldots,0)\in\Im\rho'$ and $\Im\rho^*$ commutes with
  $\Im\rho'$. Then the group $G\times G'$ with the homomorphism $\rho^*\times\rho'$
  satisfies the required conditions.
\end{proof}

From Remark~\ref{rem:G-structure} and the proof of
Corollary~\ref{cor-F-embedding} we have the following corollary.
In particular, this extends a result of Takens \cite{Takens} stating
the a formal diffeomorphism with a unipotent linear part
can always be embedded in the flow of a formal vector field.

\begin{Cor}
  Let $F\in\Diff[[\C^n,0]]$ be a formal diffeomorphism. Let
  $\l_1,\ldots,\l_n$ denote the spectrum of the linear part of $F$,
  and let $k$ denote the size of the group of roots of unity in
  the multiplicative group generated by $\l_1,\ldots,\l_n$.
  
  Then $F^k$ can be embedded in a formal flow, i.e. there exists a
  formal vector field $V\in\Der[[\C^n,0]]$ such that $e^V=F^k$.
\end{Cor}

\bibliographystyle{plain}
\bibliography{refs}

\end{document}